\documentclass[10pt]{article}
\author{Sabrina Pauli}
\date{}
\title{Computing $\A^1$-Euler numbers with Macaulay2}
\usepackage[left=2cm, right=5cm, top=2cm]{geometry}
\usepackage{amsmath}
\usepackage{amssymb}
\usepackage{amsthm}
\usepackage[all]{xy}
\usepackage{amscd}
\usepackage{tikz-cd}
\usepackage{cite}
\usepackage{enumitem}
\usepackage{fullpage}
\usepackage{verbatim}
\usepackage[onehalfspacing]{setspace}
\usepackage{xcolor}
\usepackage{appendix}

\newtheorem{theorem}{Theorem}[section]

\newtheorem{remark}[theorem]{Remark}
\newtheorem{lemma}[theorem]{Lemma}

\theoremstyle{definition}
\newtheorem{definition}[theorem]{Definition}
\newtheorem{example}[theorem]{Example}

\usepackage[utf8]{inputenc}

\newcommand{\C}{\mathbb{C}}
\newcommand{\Z}{\mathbb{Z}}
\newcommand{\A}{\mathbb{A}}
\newcommand{\PP}{\mathbb{P}}

\newcommand{\RR}{\mathbb{R}}
\newcommand{\glob}{\mathcal{O}}

\newcommand{\Gr}{\operatorname{Gr}}
\newcommand{\Sym}{\operatorname{Sym}}
\newcommand{\GW}{\operatorname{GW}}
\newcommand{\HH}{\mathbb{H}}
\newcommand{\Tr}{\operatorname{Tr}}
\newcommand{\Spec}{\operatorname{Spec}}
\newcommand{\FF}{\mathbb{F}}

\usepackage[
textwidth=3cm,
textsize=small,
colorinlistoftodos]
{todonotes}

\begin{document}
\maketitle
\begin{abstract}
We use Macaulay2 for several enriched counts in $\GW(k)$.
First, we compute the count of lines on a general cubic surface
using Macaulay2 over $\mathbb{F}_p$ in $\GW(\FF_p)$ for $p$ a prime number and over $\mathbb{Q}$ in $\GW(\mathbb{Q})$.
This gives a new proof for the fact that the $\A^1$-Euler number of $\Sym^3\mathcal{S}^*\rightarrow \Gr(2,4)$ 
 is
 $15\langle1\rangle+12\langle -1\rangle$. 
Then, we compute the count of lines in $\PP^3$ meeting 4 general lines, the count of lines on a quadratic surface meeting one general line and the count of singular elements on a pencil of degree $d$-surfaces.
Finally, we provide code to compute the EKL-form and compute several $\A^1$-Milnor numbers.

\end{abstract}
\section{Introduction}
\label{intro}
In \cite{2017arXiv170801175K} Kass and Wickelgren count the lines on a smooth cubic surface as an element of the Grothendieck-Witt ring $\GW(k)$ of a field $k$ by computing the $\A^1$-Euler number of the vector bundle $\mathcal{E}:=\Sym^3\mathcal{S}^*\rightarrow \Gr(2,4)$ which is by definition the sum of the local indices, that is the local $\A^1$-degrees, of the zeros of a general section. Here, $\Gr(2,4)$ denotes the Grassmannian of lines in $\PP^3$ and $\mathcal{S}\rightarrow \Gr(2,4)$ its tautological bundle.

For a field $L$, denote by $\mathcal{E}_L$ the base change of $\mathcal{E}$ to $L$. Let $F\in \FF_p[X_0,X_1,X_2,X_3]_3$ be a random homogeneous degree 3 polynomial in 4 variables. Then $F$ defines a general cubic surface $X=\{F=0\}\subset \PP^3_{\FF_p}$ and a section $\sigma_F$ of $\mathcal{E}_{\FF_p}$ by restriction. The zeros of $\sigma_F$ are the lines on $X$.

 Let $\A^4_{\FF_p}=\Spec(\FF_p[x_1,x_2,x_3,x_4])\subset \Gr(2,4)$ be the open affine subset of the Grassmannian consisting of the lines spanned by $x_1e_1+x_3e_2+e_3$ and $x_2e_1+x_4e_2+e_4$ where $(e_1,e_2,e_3,e_4)$ is the standard basis for $\FF_p^4$. For the general cubic surface $X$, all lines on $X$ are elements of this open affine subset of $\Gr(2,4)$ and hence the $\A^1$-Euler number $e^{\A^1}(\mathcal{E}_{\FF_p})\in \GW(\FF_p)$ (or the count of lines on the cubic surface $X$) can be computed as the sum of local $\A^1$-degrees of the zeros of $\sigma_f\vert_{\A^4}=(f_1,f_2,f_3,f_4):\A^4\rightarrow \A^4$ by \cite{MR3909901}. 

The $\FF_p$-algebra $\frac{\FF_p[x_1,x_2,x_3,x_4]}{(f_1,f_2,f_3,f_4)}$ is 0 dimensional and thus there are finitely many  lines on $X$. 
Call these lines $l_1,\dots, l_n$.
By \cite[Corollary 51]{2017arXiv170801175K} the lines on a general and thus smooth cubic surface are simple. This means that the lines $l_1,\dots,l_n$ are simple zeros of $(f_1,f_2,f_3,f_4):\A^4_{\FF_p}\rightarrow \A^4_{\FF_p}$. 
It follows that $\FF_p[x_1,x_2,x_3,x_4]/I$ is isomorphic to the product of fields $F_1\times \dots \times F_n$ where $F_j=\FF_p[x_1,x_2,x_3,x_4]/\mathfrak{m}_j$ is the field of definition of $l_j$ (that is residue field of the point in $\Gr(2,4)$ corresponding to $l_j$) for $j=1,\dots,n$. 
By \cite[Lemma 9]{MR3909901} the local index at $l_j$ is equal to $\langle J_{F_j}\rangle\in \GW(F_j)$ where $J_{F_j}$ is the image of the jacobian element $J:=\det \frac{\partial f_i}{\partial x_l}$ in $F_j=\FF_p[x,y,z,w]/\mathfrak{m}_j$ and it follows that the $\A^1$-Euler number of $\Sym^3\mathcal{S}^*\rightarrow \Gr(2,4)$ is given by
\begin{equation}
\label{eq: Euler class}
e^{\A^1}(\mathcal{E}_{\FF_p})=\sum_{j=1}^n\Tr_{F_j/\FF_p}(\langle J_{F_j}\rangle)\in \GW(\FF_p).
\end{equation}

We use Macaulay2 to compute the rank and discriminant of \eqref{eq: Euler class} when $p=32003$.
The computation gives an element in $\GW(\FF_p)$ of rank 27 and discriminant $1\in \FF_p^*/(\FF_p^*)^2$. Two elements in $\GW(\FF_p)$ are equal if and only if they have the same rank and discriminant, so this determines the count of lines on a cubic surface in $\GW(\FF_p)$ completely.
 
Similarly, we use Macaulay2 to get the Gram matrix of a form representing $e^{\A^1}(\mathcal{E}_{\mathbb{Q}})\in \GW(\mathbb{Q})$ over the rational numbers $\mathbb{Q}$. We view this form as a bilinear form over the real numbers $\RR$ and compute its signature which is equal to $3$. 

By Theorem 5.8 in \cite{bachmann2020a1euler} $e^{\A^1}(\mathcal{E})=e^{\A^1}(\Sym^3\mathcal{S}^*)$
is equal to either
\begin{equation}
\label{eq: without 2}
\frac{n_{\C}+n_{\RR}}{2}\langle1\rangle+\frac{n_{\C}-n_{\RR}}{2}\langle-1\rangle\in \GW(k)
\end{equation}
or
\begin{equation}
\label{eq: with 2}
\frac{n_{\C}+n_{\RR}}{2}\langle1\rangle+\frac{n_{\C}-n_{\RR}}{2}\langle-1\rangle +\langle 2\rangle -\langle 1\rangle\in \GW(k)
\end{equation}
for $n_{\C},n_{\RR}\in \Z$ and a field $k$.
By \cite[Remark 5.7]{bachmann2020a1euler} $n_{\C}$ and $n_{\RR}$ are the Euler numbers of the real and complex bundle, respectively. The complex count $n_{\C}$ is equal to the rank of our form which is $n_{\C}=27$, and the real count is equal to the signature, so $n_{\RR}=3$. This has already been know for a long time. The complex count $n_{\C}$ is the classical result by Cayley and Salmon that there are 27 complex lines on a smooth cubic surface \cite{cayley_2009}. Segre studied the real lines on a smooth cubic surface in \cite{MR0008171}. See also \cite{MR3370064} and \cite{MR3176620} for the real count.

Since $2$ is not a square for every prime $p$, in particular not for our chosen prime 32003, we can rule out \eqref{eq: with 2} for the count of lines on a cubic surface and hence we have a new proof of the fact that 
\begin{equation}
e^{\A^1}(\Sym^3\mathcal{S}^*)=15\langle 1\rangle +12\langle -1\rangle\in \GW(k)
\end{equation}
which is the main result in \cite{2017arXiv170801175K}. 

In \cite[$\S8$]{LEVINE_2019} and \cite[Corollary]{bachmann2020a1euler} it is shown that the $\A^1$-Euler number of direct sums of symmetric powers of the dual tautological bundle on a Grassmannian is always of form \eqref{eq: without 2} when defined, using the theory of Witt-valued characteristic classes. The proof here is independent of this theory and we may also apply it to bundles which are not of this form.

Similarly, we get an enriched count of lines meeting 4 general lines in $\PP^3$ (this has already been computed in \cite{srinivasan2018arithmetic}) and of lines on a quadratic surface meeting one general line by computing the $\A^1$-Euler numbers \\$e^{\A^1}(\bigoplus_{i=1}^4\wedge^2\mathcal{S}^*\rightarrow \Gr(2,4))$ and $e^{\A^1}(\wedge^2\mathcal{S}^*\oplus \Sym^2\mathcal{S}^*\rightarrow \Gr(2,4))$, respectively. Note, that neither of these vector bundles is a direct sum of symmetric powers of the dual tautological bundle and we cannot use \cite[$\S8$]{LEVINE_2019} and \cite[Corollary]{bachmann2020a1euler} to rule out \eqref{eq: with 2}. However, we already know that the $\A^1$-Euler number of both of these bundles will be a multiple of the hyperbolic form $\HH=\langle1\rangle+\langle-1\rangle$ since they have direct summands of odd rank \cite[Proposition 12]{srinivasan2018arithmetic}.

Furthermore, we count singular elements on a pencil of degree $d$ surfaces as the $\A^1$-Euler number of $\oplus_{i=1}^4\pi_1^*\glob_{\PP^3}(d-1)\otimes \pi_2^*\glob_{\PP^1}(1)\rightarrow \PP^3\times \PP^1$.

Finally, we provide code for computing the EKL-form (see \cite{MR3909901}) which computes the local $\A^1$-degree for non-simple zeros.

In the appendices we compute the $\A^1$-Milnor numbers of several Fuchsian singularities and provide one explicit example of the Gram matrix of a form representing $e^{\A^1}(\mathcal{E}_{\FF_{11}})\in \GW(\FF_{11})$.
\section{$\A^1$-Euler numbers}
\label{section: Euler numbers}

\subsection{Definition of the $\A^1$-Euler number}
\label{subsection: Definition Euler number}
Let $k$ be a field.
We recall the definition of the $\A^1$-Euler number from \cite[$\S4$]{2017arXiv170801175K}.
Let $\pi:E\rightarrow X$ be a vector bundle of rank $r$ over a smooth scheme $X$ over $k$. Recall that a \emph{(weak) orientation} of $E$ is an isomorphism $\phi:\det E\cong L^{\otimes2}$ where $L\rightarrow X$ is a line bundle.
\begin{definition}
A \emph{relative orientation} of $E$ is an orientation of the line bundle $\operatorname{Hom}(\det TX,\det E)$, that is, an isomorphism $\phi:\operatorname{Hom}(\det TX,\det E)\xrightarrow{\cong}L^{\otimes 2}$ where $TX\rightarrow X$ denotes the tangent bundle of $X$ and $L\rightarrow X$ is a line bundle.
\end{definition}
\begin{remark}
If both the tangent bundle of $X$ and $E$ are orientable, then $E$ is relatively orientable since $\operatorname{Hom}(\det TX,\det E)\cong (\det TX)^{-1}\otimes \det E$. However, $\pi:E\rightarrow X$ can still be relatively orientable even though $E$ and $TX$ are not orientable.  
\end{remark}
Let $\pi:E\rightarrow X$ be a vector bundle over a smooth and proper scheme $X$ over $k$ equipped with a relative orientation $\phi$ and assume that $\dim X=\operatorname{rk}E=r$. 
In our examples, $X$ is either a Grassmannian of lines or projective space which both have standard coverings by open affine spaces $U\cong \A^r$. An open affine subset $U\cong\A^r$ of $X$ defines local coordinates, that is a trivialization of $TX\vert_U$.
\begin{remark}
For general $X$ one can find `Nisnevich coordinates' around any closed point $x$ of $X$ with residue field $k(x)$ separable over $k$. These Nisnevich coordinates induce a trivialization of the tangent bundle around $x$ \cite[Definition 17 and Lemma 18]{2017arXiv170801175K}.
\end{remark}
\begin{definition}
A trivialization of $E\vert_U$ with $U\cong \A^r$ is \emph{compatible} with the relative orientation $\phi$ and the local coordinates on $X$ if the element of $\operatorname{Hom}(\det TX\vert_U,\det E\vert_U)$ sending the distinguished basis element of $\det TX\vert_U$ to the distinguished element of $\det E\vert_U$, is sent to a square by $\phi$.
\end{definition}
Let $\sigma:X\rightarrow E$ be a section of $E$ with an isolated zero $x\in X$.
Choose a neighborhood $x\in U\cong \A^r$ and a trivialization $E\vert_{U}\cong \A^r\times \A^r$ compatible with $\phi$ and the local coordinates defined by $U\cong \A^r$. Then the \emph{local index} $\operatorname{ind}_x\sigma$ at $x$ is the local $\A^1$-degree $\deg_x^{\A^1}(f_1,\dots,f_r)$ of $(f_1,\dots,f_r): \A^r\rightarrow \A^r$ at $x$ where \[(f_1,\dots,f_r):\A^r\xrightarrow{\sigma\vert_{\A^r}=(\operatorname{id},(f_1,\dots,f_r))}\A^r\times \A^r\xrightarrow{\pi_2} \A^r .\] 
The local $\A^1$-degree is the analog of the local degree in $\A^1$-homotopy theory. We do not use its formal definition and refer the interested reader to \cite[$\S2$]{MR3909901} for the definition and to \cite{2019arXiv190208857W} for an introduction to $\A^1$-homotopy theory.

Let $L/k$ be a finite separable field extension and let $\beta:V\times V\rightarrow L$ be a non-degenerate symmetric bilinear form over $L$. Then the \emph{trace form} $\Tr_{L/k}(\beta)$ is the form
\begin{equation}
\label{eq: trace form}
V\times V\xrightarrow{\beta}L\xrightarrow{\Tr_{L/k}} k
\end{equation}
where $\Tr_{L/k}$ denotes the field trace.
Assume $x\in X$ is simple zero, that is the \emph{jacobian element} $J(x):=\frac{\partial f_i}{\partial x_j}(x)$ at $x$ is non-zero.
If $x$ is a rational point, its local degree is equal to $\langle J(x)\rangle \in \GW(k)$. When $x$ is not rational, its local $\A^1$-degree can be computed as the trace form $\Tr_{k(x)/k}(\langle J(x)\rangle)\in\GW(k)$ of $\langle J(x)\rangle\in\GW(k(x))$ for finite separable field extensions $k(x)/k$ by \cite{MR3909901}.

\begin{remark}
When $x\in X$ is a non-simple zero, its local $\A^1$-degree can be computed as the \emph{EKL-form}  (see section \ref{sec: EKL-form}).
\end{remark}

We define the \emph{$\A^1$-Euler number} $e^{\A^1}(E,\sigma)$ with respect to a section $\sigma:X\rightarrow E$ with only isolated zeros to be sum of local indices of the zeros of $\sigma$. It turns out that $e^{\A^1}(E,\sigma)$ does not depend on the chosen section \cite[Theorem 1.1]{bachmann2020a1euler} and we can define the $\A^1$-Euler number independently of $\sigma$. 
\begin{definition}
Let $\pi:E\rightarrow X$ be a vector bundle of rank $r$ equal to the dimension of the smooth, proper scheme $X$ over a field $k$ equipped with a relative orientation, then the \emph{$\A^1$-Euler number} is defined by $e^{\A^1}(E):=e^{\A^1}(E,\sigma)$ for a section $\sigma$ with only isolated zeros.
\end{definition}

\subsection{Cubic Surfaces}
\label{subsection: cubic surfaces}
We compute the rank and discriminant of the $\A^1$-Euler number of $\mathcal{E}=\Sym^3\mathcal{S^*}\rightarrow \Gr(2,4)$ over $\FF_p$ with $p=32003$.
\begin{verbatim}
P = 32003 
FF = ZZ/P  
\end{verbatim}
We generate a random homogeneous degree 3 polynomial $F$ in 4 variables $X_0$, $X_1$, $X_2$ and $X_3$.
\begin{verbatim}
R = FF[X0,X1,X2,X3]
F = random(3,R) 
\end{verbatim}
We replace $X_0$, $X_1$, $X_2$ and $X_3$ by $x_1s+x_2$, $x_3s+x_4$, $s$ and $1$, respectively, and define $I$ to be the ideal in $C=\FF_p[x_1,x_2,x_3,x_4]$ generated by the coefficients $s^3$, $s^2$, $s$ and $1$ of $F(x_1s+x_2,x_3s+x_4,s,1)$.  That means, we let $\Spec C=\Spec (\FF_p[x_1,x_2,x_3,x_4])\subset \Gr(2,4)$ be the open affine subset consisting of the lines spanned by $x_1e_1+x_3e_2+e_3$ and $x_2e_1+x_4e_2+e_4$ for the standard basis $(e_1,e_2,e_3,e_4)$ of $\FF_p^4$. The monomials $s^3$, $s^2$, $s$ and $1$ define a basis of $\mathcal{E}\vert_{\Spec C}$ and thus a trivialization. This means $I$ is the ideal generated by $f_1,f_2,f_3,f_4$ where 
\[(f_1,f_2,f_3,f_4):\A^4\xrightarrow{\sigma\vert_{\A^4}=(\operatorname{id},(f_1,f_2,f_3,f_4))}\A^4\times \A^4\xrightarrow{\pi_2} \A^4,\] 
is equal to the restriction of the section $\sigma_F$ of $\mathcal{E}$ defined by $F$ to the chosen open affine subset $\Spec C=\Spec(\FF_p[x_1,x_2,x_3,x_4])$ in the chosen trivialization of $\mathcal{E}\vert_{\Spec C}$. 
\begin{remark} 
By \cite[Corollary 45]{2017arXiv170801175K} the vector bundle $\mathcal{E}$ is relatively orientable and the local coordinates defined by the subset $\Spec C\subset \Gr(2,4)$ and the trivialization $\mathcal{E}\vert _U$ defined above, are compatible with this relative orientation.
\end{remark}
\begin{verbatim}
C=FF[x1,x2,x3,x4]
S=C[s]
g={x1*s+x2,x3*s+x4,s,1} 
m=map(S,R,g) 
I=sub(ideal flatten entries last coefficients m F, C) 
\end{verbatim}
We compute the dimension and degree $rk$ of $C/I=\FF_p[x_1,x_2,x_3,x_4]/I$.
\begin{verbatim}
dim I 
rk = degree I 
\end{verbatim}
Since there are in general finitely many lines on a cubic surface, the expected dimension of $C/I$ is 0. The degree gives the rank $rk$ of \eqref{eq: Euler class} which turns out to be 27 as expected.
\begin{verbatim}
i10 : dim I

o10 = 0

i11 : rk=degree I

o11 = 27
\end{verbatim}
So $C/I$ is a finite $\FF_p$-algebra of rank 27. We know that the finitely many lines $l_1,\dots,l_n$ on $\{F=0\}\subset \PP^3$ are simple zeros of $(f_1,f_2,f_3,f_4):\A^4_{\FF_p}\rightarrow \A^4_{\FF_p}$ \cite[Corollary 51]{2017arXiv170801175K}.\\ This implies that $C/I=\FF_p[x_1,x_2,x_3,x_4]/I$ is isomorphic to the product of fields \begin{align*}
\FF_p[x_1,x_2,x_3,x_3]/\mathfrak{m}_1\times \dots \times \FF_p[x_1,x_2,x_3,x_4]/\mathfrak{m}_n= F_1\times\dots\times F_n
\end{align*} where $\mathfrak{m}_i$ is maximal ideal defining $l_i$ as point in $\Gr(2,4)$ and $F_i$ is the field of definition of $l_i$, i.e., the residue field of $l_i$ in $\Gr(2,4)$, for $i=1,\dots,n$. We use a primary decomposition of $I$ to find the $\mathfrak{m}_i$.

\begin{verbatim}
L = primaryDecomposition I
n = length L
\end{verbatim}
\begin{remark}
Since the ideals $\mathfrak{m}_i$ are actually primes, the primary ideals in the primary decomposition are the (unique) minimal primes and we can let Macaulay2 compute the minimal primes instead of the the primary decomposition of $I$. This is much more time efficient. However, for a non-smooth cubic surface, $(f_1,f_2,f_3,f_4):\A^4_{\FF_p}\rightarrow \A^4_{\FF_p}$ could have have non-simple zeros and we would need to use the primary decomposition and the EKL-form (see section \ref{sec: EKL-form}) for the non-reduced factors of $C/I$ to find the count of lines on the cubic surface.
\end{remark}
The contribution of the line $l_i$ to \eqref{eq: Euler class} is $\Tr_{F_i/\FF_p}(\langle J_{F_i}\rangle)$ where $J_{F_i}$ is the image of the jacobian element $J=\det \frac{\partial f_m}{\partial x_j}$ of $I$ in $F_i=C/\mathfrak{m}_i$. 
The discriminant of \eqref{eq: Euler class} is the product of the discriminants of the forms $\Tr_{F_i/\FF_p}(\langle J_{F_j}\rangle)$.
By \cite[Lemma 58]{2017arXiv170801175K} the discriminant of $\Tr_{F_i/\FF_p}(\langle J_{F_i}\rangle)$ is a square in $\FF_p$ if $J_{F_i}$ is a square in $F_i=\FF_p[x_1,x_2,x_3,x_4]/\mathfrak{m}_i$ when the degree $[F_i:\FF_p]$ is odd and if $J_{F_i}$ is a non-square in $F_i=\FF_p[x_1,x_2,x_3,x_4]/\mathfrak{m}_i$ when $[F_i:\FF_p]$ is even.
Since the units $\FF_q^*$ of a finite field $\FF_q$ with $q$ elements form the cyclic group of order $q-1$, $\FF_q^*/(\FF_q^*)^2$ is isomorphic to $\Z/2\Z$. By Fermat's little theorem $b^{q-1}\equiv 1 \mod q$ for $b\in \FF_q^*$ and $b$ is a square if and only if $b^{\frac{q-1}{2}}\equiv 1 \mod q$. So to find the discriminant of \eqref{eq: Euler class} we compute the product \[\operatorname{disc}(\eqref{eq: Euler class})=\prod_{i=1}^n\epsilon_iJ_{F_i}^{\frac{p^{[F_i:\FF_p]}-1}{2}}\]
where $\epsilon=-1$ when $[F_i:\FF_p]$ is even and $\epsilon=1$ when $[F_i:\FF_p]$ is odd.
\begin{verbatim}
J = determinant jacobian I
disc = 1_FF; i=0; while i<n do 
(if even degree L_i 
then 
disc=disc*lift(J_(C/L_i)^((P^(degree L_i)-1)//2),FF)*(-1)_FF
else 
disc=disc*lift(J_(C/L_i)^((P^(degree L_i)-1)//2),FF); i=i+1)
disc
\end{verbatim}
The discriminant of \eqref{eq: Euler class} is a square.
\begin{verbatim}
i17 : disc

o17 = 1
\end{verbatim}

\subsection{The trace form}
\label{subsec: trace form}

The trace form \eqref{eq: trace form} can also be defined when $L$ is a finite \'{e}tale $k$-algebra, i.e., a product $L\cong L_1\times\dots \times L_s$ of finitely many finite separable field extensions $L_1,\dots,L_s$ of $k$. 
For example $C/I=\frac{\FF_p[x_1,x_2,x_3,x_4]}{(f_1,f_2,f_3,f_4)}$ is a finite \'{e}tale algebra isomorphic to $F_1\times\dots\times F_n$ and the trace form $\Tr_{(C/I)/\FF_p}(\langle J_{C/I}\rangle)$ is equal to $\sum_{i=1}^n\Tr_{F_i/\FF_p}(J_{F_i})$ with the notation from subsection \ref{subsection: cubic surfaces} and where $J_{C/I}$ is the image of the jacobian element $J$ in $C/I$.

The following code computes the trace form $\Tr_{L/k}(\langle J_{C/I}\rangle)$ where $FF$ is a field and $I$ an ideal in a polynomial ring $C$ over $FF$ such that $C/I$ is a finite \'{e}tale algebra over $FF$ and $J\in C$.
\begin{verbatim}
traceForm = (C,I,J,FF) -> (
B:=basis(C/I);
r:=degree I;
Q:=(J_(C/I))*(transpose B)*B;
toVector := q -> last coefficients(q,Monomials=>B);
fieldTrace := q -> (M:=toVector(q*B_(0,0));i=1;while i<r do 
(M=M|(toVector (q*B_(0,i))) ; i=i+1); trace M);
matrix applyTable(entries Q, q->lift(fieldTrace q,FF)))
\end{verbatim}

\subsubsection{Lines meeting four general lines in $\PP^3$}
\label{subsection: lines meeting 4 lines}
As an example we compute the count of lines meeting 4 general lines in $\PP^3$, i.e., we compute the $\A^1$-Euler number of the bundle $\mathcal{E}_2:=\wedge^2\mathcal{S}^*\oplus\wedge^2\mathcal{S}^*\oplus\wedge^2\mathcal{S}^*\oplus\wedge^2\mathcal{S}^*\rightarrow \Gr(2,4) $. We know from \cite{srinivasan2018arithmetic} that this is equal to the \emph{hyperbolic form} $\HH:=\langle1\rangle+\langle-1\rangle$.

Clearly, $\det(\wedge^2\mathcal{S}^*\oplus\wedge^2\mathcal{S}^*\oplus\wedge^2\mathcal{S}^*\oplus\wedge^2\mathcal{S}^*)\cong (\wedge^2\mathcal{S}^*)^{\otimes 4}$  and thus the vector bundle $\mathcal{E}_2$ is orientable. The Grassmannian $\Gr(2,4)$ is orientable aswell (i.e., its tangent bundle $T\Gr(2,5)\cong \mathcal{S}^*\otimes \mathcal{Q}$ is orientable). Those two orientations yield a canonical relative orientation $\phi:\operatorname{Hom}(T\Gr(2,4),\mathcal{E}_2)\cong L^{\otimes 2}$ of $\mathcal{E}_2 $.
Over the open affine subset $\Spec(\FF_p[x_1,x_2,x_3,x_4])\subset\Gr(2,4)$ from subsection \ref{subsection: cubic surfaces}, the dual tautological bundle $\mathcal{S}^*\rightarrow \Gr(2,4)$ has basis the two monomials $s$ and $1$ (where $s$ is again the variable on the line). This basis induces a trivialization of the restriction of $\mathcal{E}_2$ to $\Spec(\FF_p[x_1,x_2,x_3,x_4])$.
By \cite[Lemma 4]{srinivasan2018arithmetic} the relative orientation $\phi$ is compatible with the chosen local coordinates and trivialization of $\mathcal{E}$.

Let $l_1,\dots,l_4$ be 4 general lines in $\PP^3$ and let $a_i,b_i$ be two independent linear forms cutting out $l_i$ for $i=1,\dots,4$.
\begin{verbatim}
a1=random(1,R)
b1=random(1,R)
a2=random(1,R)
b2=random(1,R)
a3=random(1,R)
b3=random(1,R)
a4=random(1,R)
b4=random(1,R)
\end{verbatim}
The linear forms $a_i$ and $b_i$ define a section $s_i:=a_i\wedge b_i$ of $\wedge^2\mathcal{S}^*$. A line $l$ in $\PP^3$ meets the line $l_i$ if and only if $s_i(l)=0$ by \cite[Lemma 5]{srinivasan2018arithmetic}. 
\normalsize
\begin{verbatim}
s1 = lift((last coefficients m a1)_(0,0)*(last coefficients m b1)_(1,0)
-(last coefficients m a1)_(1,0)*(last coefficients m b1)_(0,0),C)
s2 = lift((last coefficients m a2)_(0,0)*(last coefficients m b2)_(1,0)
-(last coefficients m a2)_(1,0)*(last coefficients m b2)_(0,0),C)
s3 = lift((last coefficients m a3)_(0,0)*(last coefficients m b3)_(1,0)
-(last coefficients m a3)_(1,0)*(last coefficients m b3)_(0,0),C)
s4 = lift((last coefficients m a4)_(0,0)*(last coefficients m b4)_(1,0)
-(last coefficients m a4)_(1,0)*(last coefficients m b4)_(0,0),C)
I2 = ideal(s1,s2,s3,s4)
J2 = determinant jacobian I2
traceForm(C,I2,J2,FF)
\end{verbatim}
\normalsize
Let $I_2$ be the ideal generated by the sections $s_1,\dots,s_4$ and $J_2:=\det \frac{\partial s_i}{\partial
x_j}$.
We compute the trace form $\Tr_{(C/I_2)/FF}(\langle J_2\rangle)$ (where $C$ is still $\FF_p[x_1,x_2,x_3,x_4]$) and get a form of rank 2 and discriminant $-1\in \FF_p^*/(\FF_p^*)^2$ as expected.

\subsubsection{Lines on a degree 2 hypersurface in $\PP^3$ meeting 1 general line}
\label{subsection: Lines on a degree 2 hypersurface meeting 1 line}
We compute the count of lines on a quadratic surface meeting a general line as the $\A^1$-Euler number of $\mathcal{E}_2:=\Sym^2\mathcal{S}^*\oplus \wedge^2\mathcal{S^*}\rightarrow \Gr(2,4)$.

There is a canonical isomorphism $\det (\Sym^2\mathcal{S}^*\oplus \wedge^2\mathcal{S^*})\cong (\det \mathcal{S}^*)^{\otimes 4}$. So $\mathcal{E}_3$ is orientable. Since $\Gr(2,4)$ is orientable, too (there are canonical isomorphisms $\det T\Gr(2,4)\cong\det (\mathcal{S}^*\otimes\mathcal{Q})\cong (\det \mathcal{S}^*)^{\otimes 2}\otimes(\det\mathcal{Q})^{\otimes 2}$), we get a canonical relative orientation on $\mathcal{E}_3$. 
We choose the same local coordinates on $\Gr(2,4)$ as above and find a trivialization of $\mathcal{E}_3\vert_{\Spec C}$:
As in \cite[Definition 39]{2017arXiv170801175K}
we define a basis $\tilde{e}_1=e_1$, $\tilde{e}_2=e_2$, $\tilde{e}_3=x_1e_1+x_3e_2+e_3$ and $\tilde{e}_4=x_2e_1+x_4e_2+e_4$ of $\FF_p[x_1,x_2,x_3,x_4]^4$, and let $\tilde{\phi}_1$, $\tilde{\phi}_2$, $\tilde{\phi}_3$ and $\tilde{\phi}_4$ be its dual basis. Here $e_1,e_2,e_3,e_4$ is a basis of $\FF_p^4$.
Then the open affine subset of lines spanned by $x_1e_1+x_3e_2+e_3$ and $x_2e_1+x_4e_2+e_4$, $U=\Spec \FF_p[x_1,x_2,x_3,x_4]\subset \Gr(2,4)$, yields a basis
\begin{equation}
\label{eq: basis tangent}
\tilde{\phi}_3\otimes \tilde{e}_1,\tilde{\phi}_4\otimes \tilde{e}_1,\tilde{\phi}_3\otimes \tilde{e}_2,\tilde{\phi}_4\otimes \tilde{e}_2
\end{equation}
of $T\Gr(2,4)\vert_U$ and a basis
\begin{equation}
\label{eq: basis bundle}
(\tilde{\phi}_3^2,0),(\tilde{\phi}_3\tilde{\phi}_4,0),(\tilde{\phi}_4^2,0),(0,\tilde{\phi}_3\wedge\tilde{\phi}_4)
\end{equation}
of $\mathcal{E}_2\vert_U$. 
\begin{lemma}
The coordinates defined by the basis \eqref{eq: basis tangent} (which are equal to the coordinates chosen in subsection \ref{subsection: cubic surfaces}) and the trivialization of $\mathcal{E}_3\vert_U$ defined by the basis \eqref{eq: basis bundle} are compatible with the canonical relative orientation of $\mathcal{E}_3$ described above.

\end{lemma}
\begin{proof}
The wedge product of \eqref{eq: basis tangent} is \[(\tilde{\phi}_3\wedge\tilde{\phi}_4\otimes \tilde{e}_1\wedge\tilde{e}_2)^{\otimes 2}\in(\det \mathcal{S}^*\otimes\det\mathcal{Q})^{\otimes 2}\vert_U\] and the wedge product of \eqref{eq: basis bundle} is \[(\tilde{\phi}_3\wedge\tilde{\phi}_4)^{\otimes 4}\in (\det\mathcal{S}^*)^{\otimes 4}\vert_U\] which are both squares. It follows that 
\[(\tilde{\phi}_3\wedge\tilde{\phi}_4\otimes \tilde{e}_1\wedge\tilde{e}_2)^{\otimes -2}\otimes (\tilde{\phi}_3\wedge\tilde{\phi}_4)^{\otimes 4}\in ((\det \mathcal{S}^*\otimes\det\mathcal{Q})^{\otimes -2}\otimes (\det\mathcal{S}^*)^{\otimes 4})\vert_U\]
is a square, too.
\end{proof}
We compute $e^{\A^1}(\mathcal{E}_3)$.
\begin{verbatim}
F2 = random(2,R)
a5 = random(1,R)
b5 = random(1,R)
s5 = lift((last coefficients m a5)_(0,0)*
(last coefficients m b5)_(1,0)
-(last coefficients m a5)_(1,0)*
(last coefficients m b5)_(0,0),C)
Q = sub(ideal flatten entries last coefficients m F2, C)
I3 = Q+ideal(s5)
J3 = determinant jacobian I3
traceForm(C,I3,J3,FF)
\end{verbatim}
It is a rank 4 form of discriminant $1\in \FF_p^*/(\FF_p^*)^2$. 
When we compute the form over the real numbers $\mathbb{\RR}$ (this can be done similarly as in subsubsection \ref{subsubsection: signature E}), we get a form of signature 0. Hence, we can use \cite[Theorem 5.8]{bachmann2020a1euler} to conclude that $e^{\A^1}(\mathcal{E}_3)=2\HH$.

\begin{remark}
We know already from \cite[Proposition 12]{srinivasan2018arithmetic} that we get a multiple of $\HH=\langle1\rangle+\langle-1\rangle$ for the $\A^1$-Euler number $e^{\A^1}(E\oplus E')$ when the rank of $E$ or $E'$ is odd.
\end{remark}

\subsubsection{Signature of $e^{\A^1}(\mathcal{E})=e^{\A^1}(\Sym^3\mathcal{S}^*_{\mathbb{Q}})$}
\label{subsubsection: signature E}
Let $G$ be a random degree 3 homogeneous polynomial in 4 variables with coefficients in $\mathbb{Q}$. We now compute $e^{\A^1}(\mathcal{E}_{\mathbb{Q}},\sigma_G)\in \GW(\mathbb{Q})$ where $\sigma_G$ is the section of $\mathcal{E}$ defined by $G$. Base change yields a form over $\RR$ of which we compute the signature as the number of positive eigenvalues minus the negative eigenvalues.  
\begin{verbatim}
R2 = QQ[Y0,Y1,Y2,Y3]
G=random(3,R2)
\end{verbatim}
Exactly as before, we restrict $\sigma_G:\Gr(2,4)\rightarrow \Sym^3\mathcal{S}^*$ to \\$\Spec C_2:=\Spec(\mathbb{Q}[y_1,y_2,y_3,y_4])\subset \Gr(2,4)$ and get $(g_1,g_2,g_3,g_4):\A^4_{\mathbb{Q}}\rightarrow \A^4_{\mathbb{Q}}$ and let $I_4=(g_1,g_2,g_3,g_4)$.
\begin{verbatim}
C2 = QQ[y1,y2,y3,y4]
S2 = C2[r]
g2 = {y1*r+y2,y3*r+y4,r,1}
m2 = map(S2,R2,g2)
I4 = sub(ideal flatten entries last coefficients m2 G, C2)
J4 = determinant jacobian I4
\end{verbatim}
We compute the trace form $\Tr_{(C_2/I_4)/\mathbb{Q}}(\langle (J_4)_{C_2/I_4}\rangle)$ where $J_4$ is the jacobian element of $I_4$, and get a $27\times 27$-matrix with values in $\mathbb{Q}$. Viewing it as a form over $\RR$, its signature is equal to the number of positive eigenvalues minus the number of negative eigenvalues because any real symmetric matrix can be diagonalized orthogonally.
\begin{verbatim}
Sol=traceForm(C2,I4,J4,QQ)
E=eigenvalues Sol
sgn=0;i=0;
while i<rk do(if E_i<0 then sgn=sgn-1 else sgn=sgn+1; i=i+1)
sgn
\end{verbatim}
The signature is 3.

\begin{verbatim}
i53 : sgn

o53 = 3
\end{verbatim}
So we know that the signature of $e^{\A^1}(\mathcal{E})$ is $n_{\RR}=3$ and its rank $n_{\C}=27$. Since the discriminant of $e^{\A^1}(\mathcal{E}_{\FF_p})\in \GW(\FF_p)$ with $p=32003$ is a square (and 2 is not a square in $\FF_p$), we can conclude that 
\[e^{\A^1}(\mathcal{E})=e^{\A^1}(\Sym^3\mathcal{S}^*)=15\langle1\rangle+12\langle-1\rangle\]
applying \cite[Theorem 5.8]{bachmann2020a1euler}.

\subsubsection{Singular elements on a pencil of degree $d$ hypersurfaces in $\PP^3$}
\label{subsubsection: singular elements}

Let $\{F_t=t_0F_0+t_1F_1=0\}\subset \PP^3\times \PP^1$ be a pencil of degree $d$ surfaces in $\PP^3$. A surface in the pencil is singular if there is a point on the surface on which all 4 partial derivatives vanish simultaneously. Consider the vector bundle $\mathcal{F}:=\bigoplus_{i=1}^4\pi_1^*(\glob_{\PP^3}(d-1))\otimes \pi_2^*(\glob_{\PP^1}(1))\rightarrow\PP^3\times\PP^1$ where $\pi_1:\PP^3\times\PP^1\rightarrow \PP^3$ and $\pi_2:\PP^3\times\PP^1\rightarrow \PP^1$ are the projections to the first and second factor, respectively. A pencil $X_t=\{F_t=t_0F_0+t_1F_1=0\}\subset \PP^3\times \PP^1$ defines a section $\sigma=(\frac{\partial F_t}{\partial X_0},\dots,\frac{\partial F_t}{\partial X_3})$ of this bundle where $X_0,\dots,X_3$ are the coordinates on $\PP^3$. A general singular hypersurface of degree $d$ has a unique singularity which is an ordinary double point by \cite[Proposition 7.1 (b)]{MR3617981} and, whence, the zeros of $\sigma$ are simple and count the singular elements on the pencil $X_t$. The bundle $\mathcal{F}$ is relatively orientable since $\mathcal{F}$ and $\PP^3\times \PP^1$ are orientable, and we can enrich the count of singular elements on the pencil over $\GW(k)$. 

Let $\A^3\cong U_0\subset \PP^3$ and $\A^1\cong V_0\subset\PP^1$ be the open affine subsets where the first variable does not vanish and let $\A^4\cong U:= U_0\times V_0\subset \PP^3\times \PP^1$. One can show that $U$ and the evident trivialization of $\mathcal{F}\vert_{U}$ are compatible with the relative orientation of $\mathcal{F}$ in the same manner as in \cite[Lemma 3.10]{mckean2020arithmetic}.

\begin{example}
We provide the code for $d=2$ over the field $\FF_p$.
\begin{verbatim}
F0 = random(2,R)
F1 = random(2,R)
T = R[t]
Ft = F0+t*F1
D0 = diff(X0, Ft)
D1 = diff(X1, Ft)
D2 = diff(X2, Ft)
D3 = diff(X3, Ft)
C3 = FF[x1,x2,x3,t]
m3 = map(C3,T,{t,1,x1,x2,x3})
I5 = ideal(m3 D0,m3 D1,m3 D2,m3 D3)
J5 = determinant jacobian I5
traceForm(C3,I5,J5,FF)
\end{verbatim}
\end{example} 

For the enriched count of singular elements on a pencil of degree 2 surfaces in $\PP^3$ we get a form of rank $4$, discriminant $1\in \FF_p^*/(\FF_p^*)^2$ and signature $0$, that is the form $2\HH$.
For $d=3$, i.e., the enriched count of singular elements on a pencil of cubic surfaces, we get $16 \HH$ and for $d=4$, $54 \HH$.
\begin{remark}
Again we know by \cite[Proposition 12]{srinivasan2018arithmetic} that we get a multiple of the hyperbolic form $\HH=\langle1\rangle+\langle-1\rangle$.
\end{remark}
\begin{remark}
Proposition 7.4 in \cite{MR3617981} computes the number of singular elements on a pencil of degree $d$ hypersurfaces in $\PP^n$ over the complex numbers to be $(n+1)(d-1)^n$. Whenever $n$ is odd, the corresponding bundle is relatively orientable, and the count can be enriched in $\GW(k)$ to the form $\frac{(n+1)(d-1)^n}{2}\HH$ by \cite[Proposition 12]{srinivasan2018arithmetic}. One checks that this coincides with our count for $n=3$ and $d=2,3,4$.
\end{remark}

\begin{remark}
Levine finds a formula \cite[Corollary 12.4]{levine2017enumerative} to count singular elements as the sum of $\A^1$-Milnor numbers (see subsection \ref{subsection: Milnor numbers}) of the singularities in a more general setting. It would be interesting to find an interpretation for the local indices in our count and then compare our result to Levine's count.
\end{remark}

\section{EKL-class}
\label{sec: EKL-form}
EKL is short for \emph{Eisenbud-Khimshiashvili-Levine} who computed the local degree of non-simple, isolated zeros as the signature of a certain non-degenerate symmetric bilinear form (a representative of the EKL-class) over $\RR$ in \cite{MR467800} and \cite{Khimshiashvili77}.
Eisenbud asked whether the class represented by the EKL-form which is defined in purely algebraic terms, had a meaningful interpretation over an arbitrary field $k$. His question was answered affirmatively in \cite{MR3909901} where it is shown that the EKL-class is equal to the local $\A^1$-degree.

We recall the definition of the \emph{EKL-class} from \cite{MR3909901}. 
Let $k$ be a field. Assume that $f=(f_1,\dots,f_n):\A^n_k\rightarrow \A^n_k$ has an isolated zero at the origin and let $\mathcal{Q}:=\frac{k[x_1,\dots,x_n]_0}{(f_1,\dots,f_n)}$. Define $E:=\det a_{ij}$ where the $a_{ij}\in k[x_1,\dots,x_n]$ are chosen such that $f_i=\sum_{i=1}^na_{ij}x_j$. We call $E$ the \emph{distinguished socle element} since it generates the socle of $\mathcal{Q}$ (that is the sum of the minimal nonzero ideals) when $f$ has an isolated zero at the origin \cite[Lemma 4]{MR3909901}. 
\begin{remark}
Let $J=\det\frac{\partial f_i}{\partial x_j}$ be the jacobian element. By \cite[Korollar 4.7]{MR393056} $J=\operatorname{rank}_k \mathcal{Q}\cdot E$.
\end{remark}
Let $\phi:\mathcal{Q}\rightarrow k$ be a $k$-linear functional which sends $E$ to $1$.

\begin{definition}
The \emph{EKL-class} of $f$ is the class of $\beta_{\phi}:\mathcal{Q}\times \mathcal{Q}\rightarrow k$ defined by $\beta_{\phi}(a,b)=\phi(ab)$ in $\GW(k)$.
\end{definition}

\begin{remark}
By \cite[Lemma]{MR3909901} the EKL-class is well-defined, i.e., it does not depend on the choice of $\phi$ and $\beta_{\phi}$ is non-degenerate. One can for example choose a $k$-basis $b_1,\dots,b_{n-1},E$ for $\mathcal{Q}$ and choose $\phi$ such that $\phi(b_i)=0$ for $i=1,\dots,n-1$ and $\phi(E)=1$.
\end{remark}

\subsection{EKL-code}
\label{subsection: EKL-code}
The following code computes the EKL-form of $f:\A^n\rightarrow \A^n$ with exactly one isolated zero at the origin, $\frac{k[x_1,\ldots,x_n]}{(f_1,\ldots,f_n)}$ is a local ring and when $\operatorname{char} k$ does not divide $\operatorname{rank}_k\frac{k[x_1,\ldots,x_n]}{(f_1,\ldots,f_n)}$.
\begin{remark}
This code for the EKL form only works in a very restrictive set up. In order to get code that works in greater generality one would need to localize $\frac{k[x_1,\ldots,x_n]}{(f_1,\ldots,f_n)}$ at zero. Fortunately, this has been done by now.
This is part of the new Macaulay2 package called $\mathbb{A}^1$-Brouwer degrees. 
The author thanks Giosuè Muratore for pointing this out.
\end{remark}

The input is a triple $(C,I,FF)$ where $k=FF$ is a field and the ideal $I=(f_1,\dots,f_n)\subset C=FF[x_1,\dots,x_n]$ is a complete intersection in the polynomial ring $C$.
\begin{verbatim}
EKL=(C,I,FF)->(r=degree I;
B = basis(C/I);
B2=mutableMatrix B;
J=determinant jacobian I;
toVector = q -> last coefficients(q, Monomials=>B);
E=J_(C/I)/r;
p=0;j=0; while j<r do (if (toVector E)_(j,0)!=0 then p=j;j=j+1);
B2_(0,p)=E;
B2=matrix(B2);
Q = transpose B2 * B2;
T=mutableIdentity(C/I,r);
i=0;while i<r do (T_(i,p)=(toVector E)_(i,0) ; i=i+1);
T=matrix T;
T1=T^(-1);
linear = v -> v_(p,0);
M=matrix applyTable(entries Q,q->lift(linear(T1*(toVector q)),FF));
M)
\end{verbatim}

\subsection{$\A^1$-Milnor numbers}
\label{subsection: Milnor numbers}
Kass and Wickelgren define and compute several \emph{$\A^1$-Milnor numbers} as an application of the EKL-form in \cite{MR3909901}. Let $0\in X=\{f=0\}\subset\A^n$ be a hypersurface with an isolated singularity at the origin. Then the $\A^1$-Milnor number of $X$ is 
\[\mu^{\A^1}(f):=\deg^{\A^1}_0(\operatorname{grad}(f)).\]
Kass and Wickelgren show that the $\A^1$-Milnor number is an invariant of the singularity. When $n$ is even $\mu^{\A^1}(f)$ counts the nodes to which $X$ bifurcates (see \cite{MR3909901} for more details). They compute the $\A^1$-Milnor numbers of \emph{ADE singularities}.
\subsubsection{Du Val Singularities}
We compute the EKL class of Du Val singularities, that is simple singularities in 3 variables, in Table \ref{Du Val singularities}.

\begin{table}
\caption{Du Val singularities}
\label{Du Val singularities}       
\begin{tabular}{lll}
\hline\noalign{\smallskip}
Singularity & Equation $f$ &$\mu^{\A^1}(f)=$ EKL-class of $\operatorname{grad}(f)\in \GW(\mathbb{Q})$ \\ 
\noalign{\smallskip}\hline\noalign{\smallskip}
$A_n$, $n$ odd & $x^2+y^2+z^{n+1}$ & $\frac{n-1}{2}\HH+\langle n+1\rangle$ \\ 
 $A_n$, $n$ even & $x^2+y^2+z^{n+1}$ & $\frac{n}{2}\HH$ \\
 $D_n$, $n>1$ odd & $x^2+y^2z+z^{n-1}$ & $\frac{n-1}{2}\HH+\langle-1\rangle$ \\
 $D_n$, $n$ even & $x^2+y^2z+z^{n-1}$ & $\frac{n-2}{2}\HH+\langle-1\rangle+\langle n-1\rangle$ \\
 $E_6$& $x^2+y^3+z^4$ & $3\HH$ \\
 $E_7$& $x^2+y^3+yz^3$ & $3\HH+\langle-6\rangle$ \\
  $E_8$& $x^2+y^3+z^5$ & $4\HH$ \\
\noalign{\smallskip}\hline
\end{tabular}
\end{table}

\begin{example}
As an example we provide the computation for $E_6$.
\begin{verbatim}
C4=QQ[x,y,z]
f=x^2+y^3+z^3*y
I6=ideal(diff(x,f),diff(y,f),diff(z,f))
EKL(C4,I6,QQ)
\end{verbatim}
We get the following EKL-form.
\begin{verbatim}
i74 : EKL(C4,I6,QQ)

o74 = | 0 0    0    1 0    0    0    |
      | 0 0    0    0 1/18 0    0    |
      | 0 0    0    0 0    1/18 0    |
      | 1 0    0    0 0    0    0    |
      | 0 1/18 0    0 0    0    0    |
      | 0 0    1/18 0 0    0    0    |
      | 0 0    0    0 0    0    -1/6 |

              7        7
o74 : Matrix QQ  <--- QQ
\end{verbatim}
It is easy to see that this is $3\HH+\langle-6\rangle$.
\end{example}

\begin{appendices}

\section{An example of lines on a cubic}
As an example, we provide the Gram matrix of $e^{A^1}(\mathcal{E}_{\FF_{11}},\sigma_H)$ for 
\begin{align*}H=Z_0^3-Z0^2Z_1-Z_1^2Z_2+Z_0Z_2^2-2Z_1Z_2^2-2Z_0^2Z_3-Z_0Z_1Z_3
-Z_1^2Z_3+Z_1Z_2Z_3+Z_1Z_3^2+2Z_2Z_3^2,\end{align*} 
that is, the count of lines on the cubic surface $\{H=0\}\subset \PP^3_{\FF_{11}}$. 

\begin{verbatim}
P2 = 11
FF2 = ZZ/P2
R3 = FF2[Z0,Z1,Z2,Z3]
H = Z0^3-Z0^2*Z1-Z1^2*Z2+Z0*Z2^2-2*Z1*Z2^2-
2*Z0^2*Z3-Z0*Z1*Z3-Z1^2*Z3+Z1*Z2*Z3+Z1*Z3^2+2*Z2*Z3^2
C5 = FF2[z1,z2,z3,z4]
S3 = C5[u]
g3 = {z1*u+z2,z3*u+z4,u,1}
m4 = map(S3,R3,g3)
I7 = sub(ideal flatten entries last coefficients m4 H, C5)
L2=minimalPrimes I7
n2=length L2
J7=determinant jacobian I7
\end{verbatim}
There are 5 lines on $X$. 
\begin{verbatim}
i85 : n2=length L2

o85 = 5
\end{verbatim}
Let $F_1,\dots,F_5$ be the fields of definitions of the 5 lines. We compute the trace forms of $\Tr_{F_j/\FF_{11}}\langle J_{F_j}\rangle$ for $j=1,\dots,5$ and sum them up to get $e^{A^1}(\mathcal{E}_{\FF_{11}},\sigma_F)\in \GW(\FF_{11})$.
\begin{verbatim}
Sol2 = traceForm(C5,L2_0,J7,FF2); j=1;
while j<n2 do (Sol2=Sol2++traceForm(C5,L2_j,J7,FF2);j=j+1);
Sol2
\end{verbatim}
\newpage
\small
\begin{verbatim}
i90 = Sol2

o90 = | -3 0 0  0  0  0  0  0  0  0  0  0  0  0  0  0  0  0  0  0  0  0  0  0  0  0  0  |
      | 0  1 2  0  0  0  0  0  0  0  0  0  0  0  0  0  0  0  0  0  0  0  0  0  0  0  0  |
      | 0  2 -3 0  0  0  0  0  0  0  0  0  0  0  0  0  0  0  0  0  0  0  0  0  0  0  0  |
      | 0  0 0  -3 -4 2  0  0  0  2  -2 0  0  0  0  0  0  0  0  0  0  0  0  0  0  0  0  |
      | 0  0 0  -4 5  -5 -3 3  -5 2  3  0  0  0  0  0  0  0  0  0  0  0  0  0  0  0  0  |
      | 0  0 0  2  -5 -5 -4 -4 -5 0  5  0  0  0  0  0  0  0  0  0  0  0  0  0  0  0  0  |
      | 0  0 0  0  -3 -4 5  -5 -5 5  -2 0  0  0  0  0  0  0  0  0  0  0  0  0  0  0  0  |
      | 0  0 0  0  3  -4 -5 -5 4  0  1  0  0  0  0  0  0  0  0  0  0  0  0  0  0  0  0  |
      | 0  0 0  0  -5 -5 -5 4  3  1  0  0  0  0  0  0  0  0  0  0  0  0  0  0  0  0  0  |
      | 0  0 0  2  2  0  5  0  1  -2 2  0  0  0  0  0  0  0  0  0  0  0  0  0  0  0  0  |
      | 0  0 0  -2 3  5  -2 1  0  2  -5 0  0  0  0  0  0  0  0  0  0  0  0  0  0  0  0  |
      | 0  0 0  0  0  0  0  0  0  0  0  -3 -3 2  -4 4  -1 4  0  0  0  0  0  0  0  0  0  |
      | 0  0 0  0  0  0  0  0  0  0  0  -3 2  -5 -1 2  0  0  4  0  0  0  0  0  0  0  0  |
      | 0  0 0  0  0  0  0  0  0  0  0  2  -5 4  -3 1  5  -4 1  0  0  0  0  0  0  0  0  |
      | 0  0 0  0  0  0  0  0  0  0  0  -4 -1 -3 -2 5  5  1  1  0  0  0  0  0  0  0  0  |
      | 0  0 0  0  0  0  0  0  0  0  0  4  2  1  5  -4 -1 -1 -2 0  0  0  0  0  0  0  0  |
      | 0  0 0  0  0  0  0  0  0  0  0  -1 0  5  5  -1 -1 -2 -2 0  0  0  0  0  0  0  0  |
      | 0  0 0  0  0  0  0  0  0  0  0  4  0  -4 1  -1 -2 0  2  0  0  0  0  0  0  0  0  |
      | 0  0 0  0  0  0  0  0  0  0  0  0  4  1  1  -2 -2 2  -5 0  0  0  0  0  0  0  0  |
      | 0  0 0  0  0  0  0  0  0  0  0  0  0  0  0  0  0  0  0  0  -3 -3 0  -2 -4 4  3  |
      | 0  0 0  0  0  0  0  0  0  0  0  0  0  0  0  0  0  0  0  -3 4  -4 2  2  -4 -3 2  |
      | 0  0 0  0  0  0  0  0  0  0  0  0  0  0  0  0  0  0  0  -3 -4 3  3  -4 -3 2  -5 |
      | 0  0 0  0  0  0  0  0  0  0  0  0  0  0  0  0  0  0  0  0  2  3  4  -1 -5 -2 -3 |
      | 0  0 0  0  0  0  0  0  0  0  0  0  0  0  0  0  0  0  0  -2 2  -4 -1 -4 -2 -4 -4 |
      | 0  0 0  0  0  0  0  0  0  0  0  0  0  0  0  0  0  0  0  -4 -4 -3 -5 -2 3  -4 1  |
      | 0  0 0  0  0  0  0  0  0  0  0  0  0  0  0  0  0  0  0  4  -3 2  -2 -4 -4 3  -2 |
      | 0  0 0  0  0  0  0  0  0  0  0  0  0  0  0  0  0  0  0  3  2  -5 -3 -4 1  -2 -2 |

               27        27
o90 : Matrix FF   <--- FF
\end{verbatim}

\normalsize
The sizes of the blocks are the degrees $[F_j:\FF_{11}]$ of the field extension $F_j/\FF_{11}$ for $j=1,\dots,5$. So there is one rational line on $X$, one defined over a field extension of degree 2 and 3 lines defined over a field extension of degree 8 on $X$.
\section{More $\A^1$-Milnor numbers}

We provide $\A^1$-Milnor numbers of some Fuchsian singularities (see \cite{MR1987870}) in Table \ref{Fuchsian singularities}.
\begin{table}
\caption{Fuchsian singularities}
\label{Fuchsian singularities}       
\begin{tabular}{lll}
\hline\noalign{\smallskip}
Singularity & Equation $f$ &$\mu^{\A^1}(f)=$ EKL-class of $\operatorname{grad}(f)\in \GW(\mathbb{Q})$ \\ 
\noalign{\smallskip}\hline\noalign{\smallskip}
$E_{12}$ & $x^7+y^3+z^2$ & $6\HH$ \\ 
 $Z_{11}$ & $x^5+xy^3+z^2$ & $5\HH+\langle-6\rangle$ \\
 $Q_{10}$ & $x^4+y^3+xz^2$& $5\HH$ \\
 $E_{13}$ & $x^5y+y^3+z^2$ & $6\HH+\langle-10\rangle$ \\
 $Z_{12}$ & $x^4y+xy^3+z^2$ & $5\HH+\langle-22\rangle+\langle-66\rangle$ \\
 $Q_{11}$ & $x^3y+y^3+xz^2$ & $5\HH+\langle2\rangle$ \\ 
 $W_{12}$ & $x^5+y^4+z^2$ & $6\HH$ \\
 $S_{11}$ & $x^4+y^2z+xz^2$& $5\HH+\langle-2\rangle$ \\
 $E_{14}$ & $x^8+y^3+z^2$ & $7\HH$ \\
 $Z_{13}$ & $x^6+xy^3+z^2$ & $6\HH+\langle-6\rangle$ \\
 $Q_{12}$ & $x^5+y^3+xz^2$ & $6\HH$ \\ 
 $W_{13}$ & $x^4y+y^4+z^2$ & $6\HH+\langle-2\rangle$ \\
 $S_{12}$ & $x^3y+y^2z+xz^2$& $6\HH$ \\
 $U_{12}$ & $x^4+y^3+z^3$& $6\HH$ \\
 $J_{0,3}$ & $x^9+y^3+z^2$ & $8\HH$ \\ 
 $Z_{1,0}$ & $x^7+xy^3+z^2$ & $7\HH+\langle-6\rangle$ \\
 $Q_{2,0}$ & $x^6+y^3+xz^2$& $7\HH$ \\
 $W_{1,0}$ & $x^6+y^4+z^2$ & $7\HH+\langle3\rangle$ \\
 $S_{1,0}$ & $x^5+zy^2+xz^2$ & $7\HH$ \\
 $U_{1,0}$ & $x^3y+y^3+z^3$ & $7\HH$ \\ 
 $W_{12}$ & $x^5+y^4+z^2$ & $6\HH$ \\
 $NA_{0,0}^1$ & $x^5+y^5+z^2$& $8\HH$ \\
 $VNA_{0,0}^1$ & $x^4+y^4+yz^2$ & $7\HH+\langle-2\rangle$ \\
 $J_{4,0}$ & $x^{12}+y^3+z^2$ & $11\HH$ \\
 $Z_{2,0}$ & $x^{10}+xy^3+z^2$ & $10\HH+\langle-6\rangle$ \\
 $Q_{3,0}$ & $x^9+y^3+xz^2$ & $10\HH$ \\ 
 $X_{2,0}$ & $x^8+y^4+z^2$ & $10\HH+\langle1\rangle$ \\
 $S_{2,0}^*$ & $x^7+y^2z+xz^2$& $10\HH$ \\
 $U_{2,0}^*$ & $x^6+y^3+z^3$& $10\HH$ \\
   & $x^6+y^6+z^2$ & $12\HH+\langle2\rangle$ \\
  & $x^5+y^5+xz^2$& $12\HH$ \\
  & $x^4+y^4+z^4$& $10\HH+\langle1\rangle$ \\
\noalign{\smallskip}\hline
\end{tabular}
\end{table}

\end{appendices}


%
%

\section*{Acknowledgements}
I would like to thank Kirsten Wickelgren for introducing me to the topic and her excellent guidance and feedback on this project. I am also very grateful to Anton Leykin for introducing me to Macaulay2.
I gratefully acknowledge support by the RCN Frontier ResearchGroup Project no. 250399 “Motivic Hopf Equations.”

\bibliographystyle{alpha}
\bibliography{linebib}

\end{document}